\documentclass[reqno]{amsart}
\usepackage{mathptmx}
\usepackage{pxfonts}
\usepackage{amssymb}
\usepackage{amsfonts}
\usepackage{amsmath}
\usepackage{graphicx}
\usepackage{shadow}
\usepackage{color}
\usepackage[all]{xy}
\usepackage[pagebackref]{hyperref}
\newtheorem{thm}{Theorem}[section]
\newtheorem{corollary}[thm]{Corollary}
\newtheorem{lemma}[thm]{Lemma}
\newtheorem{proposition}[thm]{Proposition}

\theoremstyle{definition}
\newtheorem{definition}[thm]{Definition}
\newtheorem{example}[thm]{Example}

\theoremstyle{remark}
\newtheorem{remark}[thm]{Remark}

\def\1{{\rm (1)}}
\def\2{{\rm (2)}}
\def\3{{\rm (3)}}
\def\4{{\rm (4)}}
\def\5{{\rm (5)}}

\makeatletter
  \newcounter{xenumi}
  
\makeatother

\begin{document}

\title[Ideal Containment in Commutative Rings]{Ideal Containment in Commutative Rings}
\thanks{$^{(\star)}$ Supported by KFUPM under DSR Research Grant \#: SB191014.}

\author[Abdeslam Mimouni]{Abdeslam Mimouni $^{(\star)}$}
\address{Department of Mathematics and Statistics, KFUPM, Dhahran 31261, KSA}
\email{amimouni@kfupm.edu.sa}

\date{\today}



\begin{abstract}
 Let $R$ be a commutative ring with identity. An ideal $I$ of $R$ is said to be a big ideal (resp. an upper big ideal) if whenever $J\subsetneqq I$ (resp. $I\subsetneqq J$), $J^{n}\subsetneqq I^{n}$ (resp. $I^{n}\subsetneqq J^{n}$) for every $n\geq 1$; and $R$ itself is a big ideal ring provided that every ideal of $R$ is a big ideal.
 In this paper we study the notions of big ideals, upper big ideals and big ideal rings in different contexts of commutative rings such us integrally closed domains, pullbacks and trivial ring extensions etc. We show that the notions of big and upper big ideals are completely different. The notion of big ideal is correlated to the notion of basic ideal and the notion of upper big ideal is correlated to the notion of $\mathcal{C}$-ideals. We give a new characterization of Pr\"ufer domains via big ideal domains and we characterize some particular cases of pullback rings that are big ideal domains. Also we give some classes of big and upper big ideals in rings with zero-divisors via trivial ring extensions.
\end{abstract}
\maketitle

\section{Introduction}
The study of powers of ideals is an important tool in commutative ring theory. Recently, a lot of interest is devoted to comparing usual algebraic and symbolic powers of ideals. However, a very little is known about comparing algebraic powers of ideals. In this regard, P. Sharma (\cite{Sharma}) introduced the notion of big ideals and big ideal rings as follows: Let $R$ be a commutative ring with identity and $J$ an ideal of $R$. Then $J$ is said to be a big ideal if whenever $I$ is a proper subideal of $J$, $I^{n}$ is a proper subideal of $J^{n}$ for every $n\geq 1$, (that is, for every ideal $I$ of $R$, $I\subsetneqq J$ implies that $I^{n}\subsetneqq J^{n}$ for every $n$), and the ring $R$ is a big ideal ring provided that every ideal is a big ideal. This paper aims at studying some ideal containment in commutative rings. First, we define a ``dual" notion to the notion of a big ideal as follows: An ideal $J$ is said to be an upper big ideal if whenever $J$ is a proper subideal of an ideal $I$ of $R$, then $J^{n}$ is a proper subideal of $I^{n}$ (that is, for every ideal $I$ of $R$, $J\subsetneqq I$ implies that $J^{n}\subsetneqq I^{n}$ for every $n$). A first motivation of our study is that the two notions of big ideal and upper big ideal are completely different, but surprisingly, a commutative ring is a big ideal ring if and only if every ideal is an upper big ideal. A second motivation is the correlation of these two notions to a well-known classes of ideals. Namely, the notions of basic ideals, $\mathcal{C}$-ideals and Ratliff-Rush closed ideals. Recall that an ideal $J\subseteq I$ is a reduction of $I$ if $JI^{n}=I^{n+1}$ for some positive integer $n$. An ideal which has no reduction other than itself is called basic. The notion of reduction was introduced by Northcott and Rees with the purpose to contribute to the analytic theory of ideals in Noetherian local rings through reductions (\cite{NR}). In \cite{H1,H2}, Hays extended the study of reductions of ideals to more general contexts of commutative rings; particularly, Pr\"ufer domains and Noetherian rings (not necessarily local). His two main results assert that ``a domain is Pr\"ufer if and only if every finitely generated ideal is basic" \cite[Theorem 6.1]{H1} and ``in an integral domain, every ideal is basic if and only if it is a one-dimensional Pr\"ufer domain" \cite[Theorem 10]{H2}.  Moreover, he showed that most results on reductions of ideals do not extend beyond the class of Noetherian rings. A kind of ``dual'' notion to the basic ideal is the notion of $\mathcal{C}$-ideal. An ideal is called a $\mathcal{C}$-ideal if it is not a reduction of any larger ideal. In \cite{Sharma}, the author proved that if $J\subseteq I$ and $I^{n}=J^{n}$ for some positive integer, then $J$ is a reduction of $I$. In particular, a basic ideal is a big ideal. Also it is clear that if $I$ is a $\mathcal{C}$-ideal, then $I$ is an upper big ideal.\\
Recall that the Ratliff-Rush closure of an ideal $I$ is defined by $I^{*}=\displaystyle\bigcup_{n\geq 1}(I^{n+1}:_{R}I^{n})=\{x\in R| xI^{n}\subseteq I^{n+1} \text{ for some } n\geq 1\}$. In the case where $R$ is a Noetherian ring, the ideal $I^{*}$ is an intersting ideal, first studied by Ratliff and Rush in \cite{RR}. In \cite{HLS}, W. Heinzer, D. Lantz and K. Shah called it the Ratliff-Rush closure of $I$, or the Ratliff-Rush ideal associated with $I$. An ideal $I$ is said to be a Ratilff-Rush ideal, or Ratliff-Rush closed, if $I=I^{*}$. Among the interesting facts of this ideal is that for any regular ideal $I$ in a Noetherian ring $R$, there exists a positive integer $n$ such that for all $k\geq n$, $I^{k}=(I^{*})^{k}$, that is, all sufficiently high powers of a regular ideal are Ratliff-Rush ideals and a regular ideal is always
a reduction of its Ratliff-Rush closure in the sense that $I(I^{*})^{n}=(I^{*})^{n+1}$ for some positive integer $n$. Also the ideal $I^{*}$ is always between $I$ and the integral closure $I'$ of $I$, that is, $I\subseteq I^{*}\subseteq I'$, where $I':=\{x\in R/ x$ satisfies an equation of the form $x^{k}+ a_{1}x^{k-1}+\dots + a_{k}=0$, where $a_{i}\in I^{i}$ for each $i\in \{1, \dots, k-1\}\}$". Thus, in a Noetherian ring, an upper big regular ideal is a Ratliff-Rush closed ideal.

\section{Motivation and General Results}\label{GR}

\begin{definition} Let $R$ be a commutative ring with identity and $I$ an ideal of $R$. We say that:\\
\1 $I$ is a (lower) big ideal if whenever $J\subsetneqq I$, $J^{n}\subsetneqq I^{n}$ for every positive integer $n$. (\cite[Definition 1.2]{Sharma}).\\
\2 $I$ is an upper big ideal if whenever $I\subsetneqq J$, $I^{n}\subsetneqq J^{n}$ for every positive integer $n$.\\
\3 $R$ is a big ideal ring if every ideal  is a big ideal (equivalently,  every ideal is an upper big ideal).
\end{definition}
The notions of big ideals and upper big ideals are completely different as it is shown by the following examples.\\

\begin{example} Example of a big ideal which is not an upper big ideal. Let $\mathbb{Q}$ be the field of rational numbers and $X$ an indeterminate over $\mathbb{Q}$. Set $V=\mathbb{Q}(\sqrt[3]{2})[[X]]=\mathbb{Q}(\sqrt[3]{2})+M$, $R=\mathbb{Q}+M$, $W=\mathbb{Q}+\sqrt[3]{2}\mathbb{Q}$ and $I=X(W+M)$. Since $W^{2}=\mathbb{Q}(\sqrt[3]{2})$, $I^{2}=X^{2}(W^{2}+M)=X^{2}V=M^{2}$. However, $I\subsetneqq M$. Thus $I$ is not an upper big ideal. Now, let $J\subsetneqq I$ and suppose that $J^{s}=I^{s}$ for some positive integer $s\geq 2$. Since $I^{2}=M^{2}$, $I^{n}=M^{n}$ for every $n\geq 2$. Then $J^{s}=I^{s}=M^{s}$. Necessarily $J$ is not an ideal of $V$ (for if $J$ is an ideal of $V$, then $J=M$ as $M$ is a big ideal of $V$). Hence $J=X(H+M)$ for some $\mathbb{Q}$-subspace $H$ of $W$. Since $J\subsetneqq I$, $H\subsetneqq W$, and since $dim_{\mathbb{Q}}W=2$, $dim_{\mathbb{Q}}H=1$. Set $H=f\mathbb{Q}$ for some $f\in H$. Then $V=(M^{s}:M^{s})=(I^{s}:I^{s})=(J^{s}:J^{s})=(H^{s}:H^{s})+M=\mathbb{Q}+M=R$, which is a contradiction. Thus $J^{n}\subsetneqq I^{n}$ for every positive integer $n$, and therefore $I$ is a big ideal.
\end{example}

\begin{example} Example of an upper big ideal which is not a big ideal. Let $K$ be a field, $X$ an indeterminate over $K$, $R=K[[X^{3}, X^{4}]]$ and $I=X^{6}k[[X]]$. Clearly $I$ is an ideal of $R$, and $K[[X]]=R+XR+X^{2}R$. Let $W=R+XR$ and $J_{0}=X^{6}W$. Since $W\subsetneqq K[[X]]$, $J_{0}\subsetneqq I$. But since $W^{2}=k[[X]]$, $J_{0}^{2}=I^{2}$. Thus $I$ is not a big ideal. Now assume that $I\subseteq J$ and $I^{n}=J^{n}$ for some positive integer $n$. Then $(X^{-6}J)^{n}=K[[X]]$, and since $K[[X]]$ is integrally closed, $X^{-6}J\subseteq K[[X]]$. Thus $J\subseteq X^{6}K[[X]]=I\subseteq J$ and hence $I=J$. Therefore $I$ is an upper big ideal, as desired.
\end{example}

\begin{example} \1 Example of a big ideal which is not basic. Let $K$ be a field and $X$ an indeterminate over $K$. Set $R=k[[X^{2}, X^{3}]]$, $T=K[[X]]$ and $M=X^{2}K[[X]]=X^{2}T$. Clearly, $X^{2}R$ is a proper reduction of $M$ and so $M$ is not a basic ideal of $R$. Moreover, $R$ is a one-dimensional Noetherian local domain with maximal ideal $M$, and by \cite[Theorem 2.2]{Mi}, $R$ is a divisorial domain. Now let $J\subseteq M$ such that $J^{n}=M^{n}$ for some $n\geq 1$. If $JJ^{-1}=R$, then $J=fR$ for some $f\in J$. But then $R=(J^{n}:J^{n})=(M^{n}:M^{n})=K[[X]]$, which is absurd. Thus $JJ^{-1}\subseteq M$, and so $T=M^{-1}\subseteq (JJ^{-1})^{-1}=(J_{v}:J_{v})=(J:J)\subseteq R'=T$. Hence $(J:J)=T$ and so $J$ is an ideal of $T$. But since $M$ is a principal ideal of $T$, $M$ is a big ideal of $T$ and so $J=M$. Thus $M$ is a big ideal of $R$, as desired.\\

\2 Example of an upper big ideal which is not a $\mathcal{C}$-ideal. With the same domain $R$, let $I=X^{2}R$. Since $I$ is a reduction of $M$, $I$ is not a $\mathcal{C}$-ideal. Let $J$ be an ideal of $R$ such that $I\subseteq J$ and $I^{n}=J^{n}$ for some positive integer $n$. If $JJ^{-1}\subseteq M$, then $T=M^{-1}\subseteq (JJ^{-1})^{-1}=(J_{v}:J_{v})=(J:J)\subseteq (J^{n}:J^{n})=(I^{n}:I^{n})=R\subsetneqq T$, which is absurd. Thus $JJ^{-1}=R$ and so $J=fR$ for some $f\in J$. Again, $I^{n}=J^{n}$ implies that $X^{2n}R=f^{n}R$ and so $(X^{2}f^{-1})^{n}R=R$. Thus $(X^{2}f^{-1})^{n}\in U(R)$ and so $X^{2}f^{-1}\in U(R)$. Thus $I=X^{2}R=fR=J$, and hence $I$ is an upper big ideal as desired.
\end{example}
Recall that the Ratliff-Rush closure of an ideal $I$ in a commutative ring $R$ is the ideal given by $I^{*}=\displaystyle\bigcup_{n\geq 0}(I^{n+1}:_{R}I^{n})$ and an ideal $I$ is said to be a Ratliff-Rush ideal if $I=I^{*}$.
\begin{proposition} Let $R$ be an integral domain. Then every Ratliff-Rush ideal is an upper big ideal.
\end{proposition}

\begin{proof} Let $I$ be a Ratliff-Rush ideal and $J$ an ideal of $R$ such that $I\subseteq J$ and suppose that $I^{s}=J^{s}$ for some $s\geq 1$. Let $x\in J^{*}$. Then $xJ^{n}\subseteq J^{n+1}$ for some $n\geq 1$. Without loss of generality, we may assume that $n\geq s$. So $xI^{n}=xJ^{n}\subseteq J^{n+1}=I^{n+1}$ and so $x\in I^{*}=I$. Thus $J\subseteq J^{*}\subseteq I\subseteq J$ and therefore $I=J$, as desired.
\end{proof}
The next example shows that an upper big ideal needs not be a Ratliff-Rush ideal.

\begin{example} Let $V$ be a valuation domain with maximal ideal $M$ such that $M=M^{2}$. Since $M$ is maximal, $M$ is an upper big ideal of $V$. However, for every $n\geq 1$, $M^{n+1}=M^{n}=M$ and so $(M^{n+1}:M^{n})=(M:M)=V$. Thus $\tilde{M}=V$ and therefore $M$ is not a Ratliff-Rush ideal.
\end{example}

 The following diagram puts the big ideals and upper big ideals of an integral domain in the perspectives. First recall that an ideal $I$ of an integral domain $R$ is stable (resp. strongly stable) if it is invertible (resp. principal) in its endomorphisms ring, that is, $I(I:I^{2})=(I:I)$ (resp. $I=a(I:I)$); and $R$ is stable (respectively strongly stable) if every ideal of $R$ is stable (resp. strongly stable).\\

  $\begin{array}{ccccccc}
  
  &\text{Principal ideal} & \Longrightarrow &\text{strongly stable ideal} & &\\
  
  &\Downarrow &  &\Downarrow &  &\\
  
   &\text{Invertible ideal}  &\Longrightarrow & \text{stable ideal} &\Longrightarrow & \text{Ratliff-Rush ideal}\\
   
   &\Downarrow &  &   &  &\Downarrow\\
   
   &\text{ Basic ideal } &\Longrightarrow &  \text{big ideal } & & \text{upper big ideal} \\
   
   & & & & &\Uparrow\\
   
  & & & & &\text{$\mathcal{C}$-ideal }\\

 \end{array}$\\
 \bigskip
 
Notice that stable domains are big ideal domains (Proposition~\ref{stable}),  and big ideal domains coincide with upper big ideal domains. Also basic ideal domains (i.e. every ideal is basic) are exactly one-dimensional Pr\"ufer domains (\cite[Theorem 10]{H2}).
We have the following diagram:\\
\bigskip

 $\begin{array}{ccccccc}
  
  \text{PID} & \Longrightarrow &\text{strongly stable domain} & &\\
  
  \Downarrow &  &\Downarrow &  &\\
  
  \text{Dedekind domain}  &\Longrightarrow & \text{stable domain} &\Longrightarrow & \text{Ratliff-Rush domain}\\
   
  \Downarrow &  &\Downarrow   &  &\Downarrow\\
   
  \text{ Basic ideal domain} &\Longrightarrow &  \text{big ideal domain } & \Longleftrightarrow& \text{upper big ideal} \\
   
 \end{array}$\\
\bigskip


Before stating our first main theorem in this section, it is worth noticing that in \cite[Theorem 14.2]{Sharma}, the author proved that in a
Noetherian domain, if every ideal is integrally closed, then $R$ is a big ideal domain. However, it is well-known that an integral domain where every
ideal is integrally closed is a Pr\"ufer domain, and thus, under the Noetherian condition, $R$ is a Dedekind domain and so a big ideal domain.

 \begin{thm}\label{GR.1} Let $R$ be an integrally closed domain. The following are equivalent.\\
 \1 Every ideal is a big ideal (i.e. $R$ is a big ideal domain).\\
 \2 Every finitely generated ideal is a big ideal.\\
 \3 $R$ is a Pr\"ufer domain.
 \end{thm}

\begin{proof} $\1 \Longrightarrow \2$ Trivial.\\
$\2\Longrightarrow \3$ Let $x={a\over{b}}\in K$ and consider the ideals $I=(a^{4}, a^{3}b, ab^{3}, b^{4})$ and $J=(a, b)^{4}=(a^{4}, a^{3}b, a^{2}b^{2}, ab^{3}, b^{4})$. Clearly $I\subseteq J$ and $I^{2}=J^{2}$.
Since $R$ is a big ideal domain, $I=J$. Thus $a^{2}b^{2}\in I$. Write $a^{2}b^{2}=\alpha_{1}a^{4}+\alpha_{2}a^{3}b+\alpha_{3}ab^{3}+\alpha_{4}b^{4}$ for some $\alpha_{1},
\alpha_{2}, \alpha_{3}$ and $\alpha_{4}$ in $R$. Dividing by $b^{4}$, we obtain an equation of the form $\alpha_{1}x^{4}+\alpha_{2}x^{3}-x^{2}+\alpha_{3}x+\alpha_{4}=0$. Thus,
by $(u, u^{-1})$-lemma (see, e.g.,  \cite[Theorem 6]{AnKw}), $R$ is a Pr\"ufer domain.\\
$\3\Longrightarrow \1$ Assume that $R$ is a Pr\"ufer domain and let $I\subsetneq J$ be ideals of $R$. Suppose that there is $n\geq 2$ such that $I^{n}=J^{n}$. Then $Max(R, I)=Max(R, J)$. Let $a\in J\setminus I$. Then there is $M\in Max(R, I)$ such that $a\not\in IR_{M}$. Since $R_{M}$ is a valuation domain, $IR_{M}\subsetneq aR_{M}$ and so $a^{-1}IR_{M}\subseteq MR_{M}$. Hence $IR_{M}\subseteq aMR_{M}$. But since $a^{n}\in J^{n}=I^{n}\subseteq I^{n}R_{M}\subseteq a^{n}M^{n}R_{M}$, $1\in M^{n}R_{M}$, which is absurd. Thus $I^{n}\subsetneq J^{n}$ for every $n\geq 2$ and therefore $R$ is a big ideal domain as desired.
\end{proof}

\begin{remark}\label{remark 1} Let $T$ is an overring of $R$ such that the conductor $(R:T)\not=0$. If $R$ is a big ideal domain, then so is $T$. Indeed, let $I\subseteq J$ be ideals of $T$ such that $I^{n}=J^{n}$ for some $n\geq 1$ and let $0\not=a\in (R:T)$. Then $aI\subseteq aJ$ are ideals of $R$ and $(aI)^{n}=(aJ)^{n}$. Thus $aI=aJ$ and therefore $I=J$.
\end{remark}

Combining the above remark and Theorem~\ref{GR.1}, we obtain the following two corollaries. By $R'$ and $\overline{R}$, we denote the integral closure and the complete integral closure of $R$ respectively.\\

\begin{corollary}\label{GR.2} Let $R$ be an integral domain such that $(R:R')\not=0$. If $R$ is a big ideal domain, then $R$ is a quasi-Pr\"ufer domain (that is, $R'$ is a Pr\"ufer domain).
\end{corollary}
Recall that an integral domain $R$ is a Mori domain if $R$ satisfies the $acc$ on divisorial ideals (see, e.g. \cite{Bar1}). Noetherian and Krull domains are Mori domains. In \cite[Theorem 1.6.2 ]{Sharma}, Sharma proved that a Noetherian big ideal ring is of Krull dimension one. Our next corollary extend this result to some classes of Mori domains.\\
\begin{corollary}\label{GR.2'} Let $R$ be a Mori domain such that $(R:\overline{R})\not=0$. If $R$ is a big ideal domain, then $dim(R)=1$.
\end{corollary}

\begin{proof} Assume that $R$ is a Mori big ideal domain and $(R:\overline{R})\not=0$. By Remark~\ref{remark 1}, $\overline{R}$ is a big ideal domain and so a Pr\"ufer domain by Theorem~\ref{GR.1} since it is integrally closed. But since $\overline{R}$ is a Krull domain (\cite[Corollary 18]{Bar2}), $\overline{R}$ is a Dedekind domain. Hence $dim(R)=dim(\overline{R})=1$ (\cite[Corollary 3.4]{BarHou}), as desired.
\end{proof}

{\bf Open Question:} Let $R$ be a Mori domain which is a big ideal domain. Is $dim(R)=1$?\\

The next theorem deals with polynomial and power series rings over an integral domain.\\

\begin{thm}\label{GR.1'} Let $R$ be an integral domain and $X$ an indeterminate over $R$. Then $R[X]$ (resp. $R[[X]]$) is a big ideal domain if and only if $R$ is a field.
\end{thm}

\begin{proof} We mimic the proof given in Theorem~\ref{GR.1}. Assume that $R[X]$ (resp. $R[[X]])$ is a big ideal domain and let $0\not=a\in R$.  Consider the ideals $J=(a^{4}, a^{3}X, aX^{3}, X^{4})$ and 
$I=(a, X)^{4}=(a^{4}, a^{3}X, a^{2}X^{2}, aX^{3}, X^{4})$. Clearly $J\subseteq I$ and $J^{2}=I^{2}$. So $J=I$; and thus $a^{2}X^{2}\in J$. Write $a^{2}X^{2}=a^{4}f_{1}+a^{3}Xf_{2}+aX^{3}f_{3}+X^{4}f_{4}$ for some $f_{1}, f_{2}, f_{3}$ 
and $f_{4}$ in $R[X]$ (resp. $R[[X]]$). Set $f_{1}=a_{0}+a_{1}X+a_{2}X^{2}+X^{3}g_{1}$ and $f_{2}=b_{0}+b_{1}X+X^{2}g_{2}$ for some $g_{1}, g_{2}$ in  $R[X]$ (resp. $R[[X]]$). Then, by equalizing the coefficients of $X^{2}$ in both sides, we obtain $a^{2}=a^{4}a_{2}+a^{3}b_{1}$. Thus $1=a(aa_{2}+b_{1})$. Hence $a\in U(R)$ and therefore $R$ is a field. 
\end{proof}
In \cite[Lemma 7.2]{Sharma}, it was proved that an invertible ideal of an integral domain is a big ideal, and in particular, a Dedekind domain is a big ideal domain.
The next proposition extends this result to stable domains. However, Example~\ref{Ex.1} shows that a (strongly) stable ideal is not necessarily a big ideal.  

\begin{proposition}\label{stable} Every stable domain is a big ideal domain.
\end{proposition}

\begin{proof} Suppose that there are ideals $I\subsetneq J$ of $R$ such that $I^{n}=J^{n}$ for some $n\geq 2$ and set $T=(I:I)$. Then
$I^{n}=II^{n-1}\subseteq JI^{n-1}\subseteq J^{n}=I^{n}$ and so $JI^{n-1}=I^{n}$. Composing both sides by $(T:I)$ ($n-1$-times) and using the fact that $I(T:I)=T$,
we obtain $JT=I$. Thus $I\subsetneq J\subseteq JT=I$, which is a contradiction. It follows that $I^{n}\subsetneq J^{n}$ for every $n\geq 1$ and therefore $R$ is a big ideal domain.
\end{proof}

For a commutative ring $R$ with identity and a positive integer $s$, we  denote by $R^{(s)}$ the product ring of $s$ copies of $R$, that is, $R^{(s)}=R\times \dots\times R=\{(x_{1}, \dots, x_{s})| x_{i}\in R\}$. Notice that  $R^{s}=R$ ($R^{s}$ here is simply the set of all finite sums of products of $s$ elements of $R$); and if $T$ is a commutative ring which is an $R$-module and $W$ is an $R$-submodule of $T$, we denote by 
$W^{q}$ the $R$-submodule of $T$ consisting of all finite sums of product of $q$ elements of $W$ (their products is considered as a product of elements of $T$).\\

\begin{proposition} Let $R$ be an integral domain, $I$ be an ideal of $R$ and $T=(I:I)$.\\
\1 Assume that $I$ is strongly stable. Then $I$ is a big ideal if and only if for every $R$-submodule $W$ of $T$, $W^{n}\subsetneq T$ for every $n\geq 2$.\\
\2 Assume that $I$ is stable and let $\{a_{1}, \dots, a_{s}\}$ be a minimal generating set of $I$. Then $I$ is a big ideal if and only if for every $R$-submodule $W$ of $T^{(s)}$, $W^{q}=T^{(s)}$ for some $q\geq 1$ implies that $W=T^{(s)}$.
\end{proposition}

\begin{proof} \1 Set $I=aT$ and assume that $I$ is a big ideal. Let $W$ be an $R$-submodule of $T$ and set $J=aW$. Then $J\subseteq I$ and so for every $n\geq 2$, $J^{n}=a^{n}W^{n}\subsetneq I^{n}=a^{n}T$. Thus $W^{n}\subsetneq T$ as desired.\\
Conversely, let $J\subsetneq I$ and set $W=\{x\in T| xa\in J\}$. Clearly $W$ is an $R$-submodule of $T$ and $J=aW$. Since $W^{n}\subsetneq T$, $J^{n}=a^{n}W^{n}\subsetneq a^{n}T=I^{n}$ and therefore $I$ is a big ideal of $R$.\\

\2 Assume that $I$ is stable and let $\{a_{1}, \dots, a_{s}\}$ be a minimal generating set of $I$. Suppose that $I$ is a big ideal and let $W\subseteq T^{(s)}$ be an $R$-module such that $W^{q}=T^{(s)}$ for some positive integer $q$. Let $J$ be the ideal of $R$ generated by all elements of the form $\displaystyle\sum_{i=1}^{i=s}x_{i}a_{i}$ where $(x_{1}, \dots, x_{s})\in W$. Clearly $J\subseteq I$ and  $J^{q}=I^{q}$. Thus $J=I$ and so $W=T^{s}$, as desired.\\
Conversely, let $J\subseteq I$ such that $J^{q}=I^{q}$ for some $q\geq 1$. Then $J$ is a reduction of $I$ and since $I$ is stable, $JT=I$. Let $W=\{(x_{1}, \dots, x_{s})\in T^{(s)}|\displaystyle\sum_{i=1}^{i=s}x_{i}a_{i}\in J\}$. Then it is easy to check that $W$ is an $R$-submodule of $T^{(s)}$, and since $J^{q}=I^{q}$, $W^{q}=T^{(s)}$. Thus, by hypothesis, $W=T^{(s)}$ and so $J=I$. Therefore $I$ is a big ideal as desired.
\end{proof}
\section{pullback constructions}\label{Pull}

Let $T$ be a domain, $M$ a maximal ideal of $T$, $K$ its residue field, $\phi:T\longrightarrow K$ the canonical surjection, $D$ a proper subring of $K$, and $k:=qf(D)$. Let $R$ be the pullback issued from the following diagram of canonical homomorphisms:
\[\begin{array}{cccl}
                    &R:=\phi^{-1}(D) & \longrightarrow                       & D\\
(\ \square\ )       &\downarrow         &                                       & \downarrow\\
                    &T                  & \stackrel{\phi}\longrightarrow     & K=T/M.
\end{array}\]
Clearly, $M=(R:T)$ and $D\cong R/M$. For ample details on the ideal structure of $R$ and its ring-theoretic properties, we refer the reader to \cite{ABDFK, AD, BG, BR, F, FG, GH, HH, Kab, KLM}; and for more details on the constructions $D+XK[X]$ and rings between $D[X]$ and $K[X]$, we refer to \cite{CMZ, Zaf1, Zaf2, Zaf3} .

\begin{thm}\label{Pull.1} For the diagram of type $(\square)$,\\
\1 If $R$ is a big ideal domain, then $D$ and $T$ are big ideal domains.\\
\2 Assume that $R$ is a big ideal domain. If $K$ is algebraic over $k$, then $[K:k]=2$.\\
\3 Assume that $T=V$ is a valuation domain. Then $R$ is a big ideal domain if and only if $D$ is a big ideal domain and for every $D$-submodules $H\subsetneq W\subsetneq K$, $H^{n}\subsetneq W^{n}\subsetneq K$ for every $n\geq 1$.\\
\4 Assume that $T$ is a valuation domain, $D$ is a conducive domain and $qf(D)=K$. Then $R$ is a big ideal domain if and only if $D$ is a big ideal domain.
\end{thm}

\begin{proof} \1 Since $(R:T)=M$, then $T$ is a big ideal domain. Now let $A\subsetneq B$ be ideals of $D$. Then $\phi^{-1}(A)\subsetneq \phi^{-1}(B)$ are ideals of $R$ and
so for every $n\geq 1$, $\phi^{-1}(A^{n})=(\phi^{-1}(A))^{n}\subsetneq (\phi^{-1}(B))^{n}=\phi^{-1}(B^{n})$. Thus $A^{n}\subsetneq B^{n}$ and therefore $D$ is a big ideal domain.\\
\2 Assume that $K$ is algebraic over $k$. Claim: For every $\lambda\in K\setminus k$, $[k(\lambda):k]=2$. Indeed, set $r=[k(\lambda):k]$ and let $W=k+k\lambda$, $I=a\phi^{-1}(W)$ and $J=a\phi^{-1}(k(\lambda))$ for some $0\not=a\in M$. Clearly $I\subseteq J$ and since $W^{r-1}=k(\lambda)$, $I^{r-1}=J^{r-1}$. Since $R$ is a big ideal domain,  $I=J$ and so $W=k(\lambda)$. Thus $r=[k(\lambda):k]=2$. Now suppose that $k(\lambda)\subsetneq K$ and let $\mu\in K\setminus k(\lambda)$. By the Claim $[k(\mu):k]=2$. Let $W=k+k\lambda+k\mu$, $I=a\phi^{-1}(W)$ and $J=a\phi^{-1}(k(\lambda, \mu)$ for some $0\not=a\in M$. Since $W^{2}=k(\lambda, \mu)$, $I^{2}=J^{2}$. Thus $I=J$ and so $W=k(\lambda, \mu)$. Hence $3=dim_{k}W=[k(\lambda, \mu):k]=4$ which is a contradiction. It follows that $K=k(\lambda)$ and $[K:k]=2$.\\
\3 Assume that $T=V$ is a valuation domain and $R$ is a big ideal domain. By \1, $D$ is a big ideal domain. Furthermore, if there are $D$-submodules $H\subsetneq W\subsetneq K$ such that $H^{n}=W^{n}$ (or $W^{n}=K$ for some $n\geq 2$), then the ideals $I=a\phi^{-1}(H)$ and $J=a\phi^{-1}(W)$ (or $I=a\phi^{-1}(W)$ and $J=aV$) satisfy $I\subsetneq J$ but $I^{n}=J^{n}$, which is a contradiction. Hence for every $D$-submodules $H\subsetneq W\subsetneq K$, $H^{n}\subsetneq W^{n}\subsetneq K$ for every $n\geq 1$.\\
Conversely, let $I\subsetneq J$ be ideals of $R$. If $I$ and $J$ are ideals of $V$, then $I^{n}\subsetneq J^{n}$ for every $n\geq 1$ as $V$ is a big ideal domain. If $I$ is an ideal of $V$, necessarily $J$ is not an ideal of $V$. Two cases are then possible:\\
Case 1, $J\subsetneq M$. Then $J=a\phi^{-1}(W)$ for some $D$-submodule $W$ of $K$. Since $I\subsetneq J=a\phi^{-1}(W)$, $a^{-1}I\subsetneq \phi^{-1}(W)\subsetneq V$. But since $I$ is an ideal of $V$, $a^{-1}I\subseteq M$ and so $I\subseteq aM$. Thus for every $n\geq 2$, $I^{n}\subseteq (aM)^{n}=a^{n}M^{n}\subsetneq J^{n}$.\\
Case 2, $M\subsetneq J$. Then $J=\phi^{-1}(A)$ for some nonzero ideal $A$ of $D$ and so for every $s\geq 1$, $M\subsetneq J^{s}$. Since $I$ is an ideal of $V$, $I\subseteq M\subsetneq J$ and so $I^{n}\subseteq M^{n}\subseteq M\subsetneqq J^{n}$ for every $n\geq 1$, as desired.\\
Assume that $I$ is not an ideal of $V$ and set $I=b\phi^{-1}(F)$ for some $D$-submodule $F$ of $K$ and $0\not=b\in M$. If $J$ is an ideal of $V$ and $I^{n}=J^{n}$ for some $n\geq 2$,
then $b^{n}V=I^{n}V=I^{n}=J^{n}$. So $b^{n}\phi^{-1}(F^{n})=I^{n}=J^{n}=b^{n}V$. Thus $F^{n}=K$, which is absurd. Hence $I^{n}\subsetneq J^{n}$ for every $n\geq 1$.\\
Finally, assume that $J$ is not an ideal of $V$. Then $J=a\phi^{-1}(W)$ for some $D$-submodule $W$ of $K$. Since $I\subsetneq J$, we may assume that $I=a\phi^{-1}(H)$ where $H\subsetneq W\subsetneq K$. By hypothesis, $H^{n}\subsetneq W^{n}$ for every $n\geq 1$ and therefore $I^{n}\subsetneq J^{n}$ for every $n\geq 1$, as desired.\\
\4 Assume that $D$ is a conducive domain, $qf(D)=K$ and $D$ is a big ideal domain. Then for every $D$-submodules $H\subsetneq W\subsetneq K$, there is $0\not=d\in D$ such that $dH\subsetneq dW\subseteq D$ and so for every $n\geq 2$, $(dH)^{n}\subsetneq (dW)^{n}$. Thus $H^{n}\subsetneq W^{n}\subsetneq K$. By \3, $R$ is a big ideal domain. (Notice that if $W^{s}=K$ for some $s$, then $K=d^{s}W^{s}=(dW)^{s}\subseteq D$ and so $K=D$ which is absurd).
\end{proof}
\begin{corollary}\label{Pull.2} Let $R$ be a $PVD$, $M$ its maximal ideal, $k=R/M$ its residue field, $V$ its associated valuation overring and $K=V/M$. Then $R$ is a big
ideal domain if and only if $[K:k]=2$.
\end{corollary}

\begin{proof} Assume that $R$ is a big ideal domain and let $\overline{k}$ be the algebraic closure of $k$ in $K$. Then $R'=\phi^{-1}(\overline{k})$. Since $(R:R')=M$, $R'$ is a Pr\"ufer domain (Corollary~\ref{GR.2}). Thus $\overline{k}=K$ and so $K$ is algebraic over $k$. By Theorem~\ref{Pull.1}, $[K:k]=2$ as desired.\\
The converse follows easily from Theorem~\ref{Pull.1}\3.
\end{proof}


Combining Theorem~\ref{GR.1'} and Theorem~\ref{Pull.1}, we obtain the following corollary.\\
\begin{corollary} Let $A\subsetneqq B$ be an extension of integral domains, $X$ an indeterminate over $A$ and $R=A+XB[X]$ (resp. $R=A+XB[[X]]$). If $R$ is a big ideal domain, then $A$ is a big ideal domain and $B=K$ is a field.
\end{corollary}

\begin{proof} Assume that $R$ is a big ideal domain and set $T=B[X]$ (resp. $T=B[[X]]$). Since $(R:T)\not=(0)$, $T$ is a big ideal domain (Remark~\ref{remark 1}) and by Theorem~\ref{GR.1'}, $B=K$ is a field. The conclusion follows from Theorem~\ref{Pull.1}.
\end{proof}

\section{Big and upper big ideals in rings with zero-divisors:Trivial Ring Extensions}\label{TRE}

Let $A$ be a commutative ring and $E$ an $A$-module. The trivial ring extension of $A$ by $E$ (also called the idealization of $E$ over $A$) is the ring $R=A\ltimes E$  whose underlying group is $A\times E$ with multiplication given by $(a,e)(a',e') = (aa',ae'+a'e)$.  It was introduced by Nagata \cite{Na} in order to facilitate interaction between rings and their modules as well as provide diverse contexts of rings with zero-divisors. Recall that if $I$ is an ideal of $A$ and $F$ is a submodule of $E$, then $I\ltimes F$ is an ideal of $R$ if and only if $IE\subseteq F$; however, ideals of $R$ need not be of this form \cite[Example 2.5]{KMa}. Also notice that prime (resp., maximal) ideals of R have the form $P\ltimes E$, where $P$ is a prime (resp., maximal) ideal of $A$  \cite[Theorem 3.2]{AW}.  Suitable background on commutative trivial ring extensions is \cite{AW, Gl}.  Let $L$ be an ideal of $R$ and set $I_{L}=\{a\in A| (a, e)\in L \text{ for some } e\in E\}$ and $F_{L} =\{e\in E| (a, e)\in L \text{ for some } a\in A\}$. Then a routine verification shows that $I_{L}$ is an ideal of $A$, $F_{L}$ is a submodule of $E$ with $I_{L}E\subseteq F_{L}$ (so that $I_{L}\ltimes F_{L}$ is an ideal of $R$) and $L\subseteq I_{L}\ltimes F_{L}$. Notice that since $L$ is a proper ideal of $R$, $I_{L}$ is a proper ideal of $A$. We shall refer to $I_{L}\ltimes F_{L}$ as a decomposition of $L$.\\

Before stating the main theorem of this section, we need the following useful lemma.

\begin{lemma}\label{TRE.01}
Let $A$ be a commutative ring, $E$ a simple $A$-module, $R:=A\ltimes E$, $L$ a nonzero ideal of $R$ and $I_{L}\times F_{L}$ a decomposition of $L$. Then $L^{2}=(I_{L}\ltimes F_{L})^{2}$.
\end{lemma}

\begin{proof} If $I_{L}=0$, necessarily $F_{L}=E$. Then $L^{2}\subseteq (I_{L}\ltimes F_{L})^{2}=(0\ltimes E)^{2}=0\subseteqq L^{2}$. Thus $L^{2}=(0\ltimes E)^{2}$. Assume that $I_{L}\not=0$. Let $(a, e)$ and $(b, f)$ any elements in $I_{L}\times F_{L}$ and let $g, h\in F_{L}$ such that $(a, g), (b, h)\in L$. Since $E$ is a simple $A$-module, either $F_{L}=0$ or $F_{L}=E$. If $F_{L}=0$, necessarily $e=f=g=h=0$ and so $(a, e)(b, f)=(a, 0)(b, 0)=(ab, 0)\in L^{2}$. Assume that $F_{L}=E$.\\
Claim: Let $0\not=x\in I_{L}$ such that $xE=E$. Then $(x, q)\in L$ for every $q\in E$. Indeed, let $q\in E$ and let $w\in F_{L}=E$ such that $(x, w)\in L$. Since $xE=E$, $q-w=x\mu$ for some $\mu\in E$. Then $(x, q)=(x, w)(1, \mu)\in L$, as desired. Now, three cases are then possible.\\
Case 1. $aE=bE=0$. Then $(a, e)(b, f)=(ab, 0)=(a, g)(b, h)\in L^{2}\subseteqq L$.\\
Case 2. $aE=bE=E$. By the claim, $(a, e)$ and $(b, f)$ are in $L$ and so $(a, e)(b, f)\in L^{2}$.\\
Case 3. $aE=0$ and $bE=E$. By the claim $(b, f)\in L$. Since $e, g\in E=bE$, $e=bq$ and $g=bp$ for some $q, p\in E$. Then $(a, e)(b, f)=(ab, be)=(ab, b^{2}q)=(ab, bg)+(0, b^{2}q-bg)=(ab, bg)+(0, b^{2}q-b^{2}p)$. Since $(ab, bg)=(a, g)(b, f)\in L^{2}$ and $(0, b^{2}q-b^{2}p)=(b, f)(b, f)(0, q-p)\in L^{2}$, $(a, e)(b, f)\in L^{2}$. It follows that $(I_{L}\ltimes F_{L})^{2}=L^{2}$, as desired.
\end{proof}
\begin{thm}\label{TRE.1} Let $A$ be a commutative ring, $E$ an $A$ module and $R=A\ltimes E$.\\
\1 Let $I$ be an ideal of $A$ and $F$ a submodule of $E$ such that $IE\subseteq F$. If $I\ltimes F$ is a big ideal of $R$, then $I$ is a big idealof $A$.\\
\2 Assume that $A$ is an integral domain and $E$ is a divisible $A$-module. Then every non-nilpotent ideal of $R$ is a big ideal if and only if $A$ is a big ideal domain.\\
\3 Let $I$ be an ideal of $A$ such that $IE=0$. Then $I\ltimes E$ is never a big ideal of $R$.\\
\4 Let $I$ be an ideal of $A$. Then $I\ltimes E$ is an upper big ideal of $R$ if and only if $I$ is an upper big ideal of $A$.\\
\end{thm}

\begin{proof}  \1 Assume that $I\ltimes F$ is a big ideal of $R$ and let $J\subseteq I$ such that $J^{n}=I^{n}$ for some positive integer $n$. Since $J\ltimes F\subseteq I\ltimes F$ and 
$(J\ltimes F)^{n+1}=J^{n+1}\ltimes J^{n}F=I^{n+1}\ltimes I^{n}F=(I\ltimes F)^{n+1}$, $J\ltimes F=I\ltimes F$. Thus $J=I$ and so $I$ is a big ideal of $A$.\\
\2 First notice that every non-nilpotent ideal of $R$ is of the form $I\ltimes E$ for some nonzero ideal $I$ of $A$ (\cite[Corollary 3.4]{AW}). Moreover, for every nonzero ideals $I$ and $J$ of $A$, $(I\ltimes E)(J\ltimes E)=IJ\ltimes E$. Thus if $(I\ltimes E)\subsetneq (J\ltimes E)$, then $I\subsetneq J$. Since $A$ is a big ideal domain, $I^{n}\subsetneq J^{n}$ for every $n\geq 2$. Thus $(I\ltimes E)^{n}=I^{n}\ltimes E\subsetneq J^{n}\ltimes E=(J\ltimes E)^{n}$, as desired.\\
\2 Assume that $IE=0$. Then $(I\ltimes 0)\subsetneqq (I\ltimes E)$ and $(I\ltimes 0)^{2}=I^{2}\ltimes 0=(I\ltimes E)^{2}$, as desired.\\
\3 Assume that $I$ is an upper big ideal of $A$ and let $L$ be an ideal of $R$ such that $(I\ltimes E)^{n}=L^{n}$ for some positive integer $n$. Let $I_{L}\ltimes F_{L}$ be the decomposition of $L$. Since $I\ltimes E\subseteq L\subseteq I_{L}\ltimes F_{L}$, $E=F_{L}$. Now, let $a_{1}, \dots, a_{n}$ any arbitrary elements in $I_{L}$ and let $e_{1}, \dots, e_{n}$ in $E$ such that $(a_{i}, e_{i})\in L$ for every $i=1\dots, n$. Then $\displaystyle\prod_{i=1}^{n}(a_{i}, e_{i})\in L^{n}=(I\ltimes E)^{n}$, and so $\displaystyle\prod_{i=1}^{n}a_{i}\in I^{n}$. Thus $I_{L}^{n}=I^{n}$ and so $I=I_{L}$. Therefore $I\ltimes E\subseteq L\subseteq I_{L}\ltimes E=I\ltimes E$ and so $I\ltimes E=L$. Hence $I\ltimes E$ is an upper big ideal of $R$.\\
Conversely, assume that $I\ltimes E$ is an upper big ideal of $R$ and let $J$ be an ideal of $A$ such that $I\subseteq J$ and $I^{n}=J^{n}$ for some positive integer $n$. Then $(I\ltimes E)^{n+1}=I^{n+1}\ltimes I^{n}E=J^{n+1}\ltimes J^{n}E=(J\ltimes E)^{n+1}$, and so $I\ltimes E=J\ltimes E$. Hence $I=J$, as desired.
\end{proof}
\begin{corollary}\label{TRE.2} Let $A$ be a commutative ring, $I$ a nonzero ideal of $A$, $E=A/I$ and $R:=A\ltimes E$. Then $R$ is not a big ideal ring.
\end{corollary}

\begin{proof} Since $IE=0$, $I\ltimes 0$ is a proper subideal of $I\ltimes E$, but $(I\ltimes 0)^{2}=I^{2}\ltimes 0=(I\ltimes E)^{2}$. Thus $R$ is not a big ideal ring.
\end{proof}
A special case of trivial ring extension that is often studied is the trivial ring extension $R:=A\ltimes (A/M)$ where $A$ is a commutative ring and $M$ is a maximal ideal of $A$. One can put this type of trivial ring extension in a more general context of $R=A\ltimes E$ where $E$ is an $A$-module with $ME=0$ ($M$ is a maximal ideal of $A$). The next theorem characterizes big and upper big ideals in this special case of trivial ring extension, and lead us to construct upper big ideals that are not big ideals in rings with zero-divisors.  The key is the fact that for such a trivial ring extension, $L^{2}=(I_{L}\ltimes F_{L})^{2}$ for every ideal $L$ of $R$ and $I_{L}\times F_{L}$ the decomposition of $L$.\\

\begin{thm}\label{TRE.4}
Let $A$ be a commutative ring, $M$ a maximal ideal of $A$, $E$ an $A$-module such that $ME=0$, $R:=A\ltimes E$, $I$ an ideal of $A$ and $L$ a nonzero ideal of $R$. Then:\\
\1 $L^{2}=(I_{L}\ltimes F_{L})^{2}$. In particular $L$ is an upper big ideal of $R$ if and only if $F_{L}=E$ and $I_{L}$ is an upper big ideal of $A$.\\
\2 Assume that $I\not\subseteq M$. Then $I\ltimes E$ is a big ideal of $R$ if and only if $I$ is a big ideal of $A$.\\
\3 Assume that $I\subseteq M$. Then $I\ltimes E$ is not a big ideal.
\end{thm}

\begin{proof} \1 Let $L$ be a nonzero ideal of $R$ and $I_{L}\times F_{L}$ be a decomposition of $L$. If $I_{L}=(0)$, then $L^{2}=(I_{L}\ltimes F_{L})^{2}=(0)$. So without loss of generality we may assume that $I_{L}\not=(0)$. Let $(a, e), (b, f)\in I_{L}\ltimes F_{L}$, and let $g, h\in F_{L}$ such that $(a, g), (b, h)\in L$. Four subcases are possible:\\

Case 1 $a$ and $b$ are in $M$. Then $(a, f)(b, f)=(ab, 0)=(a, g)(b, h)\in L^{2}$.\\

Case 2 $a\not\in M$ and $b\not\in M$. Then $ab\not\in M$, and so there are $d\in A$ and $t\in M$ such that $1=abd+t$. Since $ME=0$, $[(af+be)-(ah+bg)]t=0$ and so $(af+be)-(ah+bg)=abd[(af+be)-(ah+bg)]$. Set $\mu=d[(af+be)-(ah+bg)]\in F_{L}$. Then $ab\mu=(af+be)-(ah+bg)$ and clearly $(a, e)(b, f)=(ab, af+be)=(ab, ah+bg)(1, \mu)=(a, g)(b, h)(1, \mu)\in L^{2}$.\\

Case 3 $a\in M$ and $b\not\in M$. Let $\alpha\in A$, $m\in M$ such that $b\alpha+m=1$. Clearly, $(a, e)(b, f)=(ab, be)$ and $(a, g)(b, h)=(ab, bg)$ (since $af=ah=0$ as $a\in M$). Notice that $(0, be)=(b, h)(0, e)\in L$. Also since $ME=0$, $be=b^{2}\alpha e$, and so $(0, be)=(0, b^{2}\alpha e)=(0, be)(b, h)(\alpha, 0)\in L^{2}$. Since $(ab, be)=(ab, 0)+(0, be)$ and $(0, be)\in L^{2}$, all what we need is to prove that $(ab, 0)\in L^{2}$. First we have $ab=ab^{2}\alpha+abm$. So $(ab, 0)=(ab^{2}\alpha, 0)+(abm, 0)$. Since $(abm, 0)=(ab, bg)(m, 0)=(a, g)(b, h)(m, 0)\in L^{2}$, we need to show that $(ab^{2}\alpha, 0)\in L^{2}$. For this, clearly $I_{L}\not\subseteq M$ (since $b\in I_{L}\setminus M$), and so $I_{L}+M=A$. Let $c\in I_{L}$, $d\in M$ such that $c+d=1$ and let $q\in F_{L}$ such that $(c, q)\in L$. Then $ab^{2}\alpha=ab^{2}\alpha c+ab^{2}\alpha d$. So $(ab^{2}\alpha, 0)=(ab^{2}\alpha c, 0)+(ab^{2}\alpha d, 0)=(b, h)(c, q)(ab\alpha, 0)+(ab, bg)(b\alpha d, 0)\in L^{2}$ (since $ab$ and $d$ are in $M$, and so $abE=dE=0$).\\

Case 4 $a\not\in M$ and $b\in M$. Just change the role of $a$ and $b$ in case 3.\\

Thus, in all cases, $(a, e)(b, f)\in L^{2}$. Hence $L^{2}=(I_{L}\ltimes F_{L})^{2}$, as desired.\\

Now, Assume that $L$ is an upper big ideal of $R$. Then $L=I_{L}\ltimes F_{L}$. If $I_{L}\not\subseteq M$, then $I_{L}+M=A$. Let $b\in I_{L}$ and $m\in M$ such that $1=b+m$. Then for every $e\in E$, $e=1.e=be+me=be\in I_{L}E\subseteq F_{L}$ (since $me\in ME=0$). Thus $E=F_{L}$. If $I_{L}\subseteq M$, then $L^{2}=(I_{L}\ltimes F_{L})^{2}=I_{L}^{2}\ltimes 0=(I_{L}\ltimes E)^{2}$. Thus $L=I_{L}\ltimes F_{L}=I_{L}\ltimes E$ and so $E=F_{L}$. Therefore $L=I_{L}\ltimes E$.\\
Now, $I_{L}\ltimes E$ is an upper big ideal of $R$ if and only if $I_{L}$ is an upper big ideal of $A$ follows from Theorem~\ref{TRE.1}(4).\\

\2 Assume that $I\not\subseteq M$ and $I$ is a big ideal of $A$. Since $I+M=A$, for every submodule $F$ of $E$, $F=AF=(I+M)F=IF$ (as $MF=0$). So by induction on $s$, $F=I^{s}F$ for every positive integer $s$. Let $L\subseteq I\ltimes E$ such that $L^{n}=(I\ltimes E)^{n}$ for some positive integer $n$. Since $L\subseteq I_{L}\ltimes F_{L}\subseteq I\ltimes E$, $I^{n}\ltimes E=I^{n}\ltimes I^{n-1}E=(I\ltimes E)^{n}=L^{n}=(I_{L}\ltimes F_{L})^{n}=I_{L}^{n}\ltimes I_{L}^{n-1}F_{L}=I_{L}^{n}\ltimes F_{L}$. Then $I^{n}=I_{L}^{n}$ and so $I=I_{L}$. Therefore $E=F_{L}$. Now, let $(a, e)\in I\ltimes E$ and let $f\in E$ such that $(a, f)\in L$. Let $c\in I$ and $m\in M$ such that $1=c+m$ (since $I+M=A$). Let $g\in E$ such that $(c, g)\in L$. Clearly $e-f=1.(e-f)=(c+m)(e-f)=c(e-f)$ (as $m(e-f)=0$) and so $(0, e-f)=(0, c(e-f))=(c, g)(0, e-f)\in L$. Thus $(a, e)=(a, f)+(0, e-f)\in L$ and so $I\ltimes E=L$. Therefore $I\ltimes E$ is a big ideal of $R$, as desired.\\
\3 Follows from Theorem~\ref{TRE.1}(3) since $IE=0$. 
\end{proof}

We close this section by the following remark.

\begin{remark} \1 If $R$ is a commutative ring (with zero-divisors) and $I$ is a non-zero nilpotent ideal of $R$ (that is, $I^{n}=0$ for some positive integer $n$), then $I$ is not a big ideal of $R$. However, a nonzero nilpotent ideal may be an upper big ideal (see Example~\ref{Example 7}).\\
\2 If $A$ and $B$ are commutative rings and $R=A\times B$, then $R$ is a big ideal ring if and only if $A$ and $B$ are big ideal rings. In particular, if $A$ and $B$ are big ideal domains, then $R$ is a big ideal ring with zero-divisors but with no nonzero nilopotent ideals.
\end{remark} 
\section{examples}\label{Ex}

In \cite[Theorem 8.3]{Sharma}, it was proved that a maximal ideal in a regular ring is a big ideal. The following example shows that a maximal ideal which is a strongly stable ideal in a Noetherian domain $R$ is not necessarily a big ideal.\\

\begin{example}\label{Ex.1} \1 Let $\mathbb{Q}$ be the field of rational numbers, $X$ an indeterminate over $\mathbb{Q}$, $R=\mathbb{Q}+X\mathbb{Q}(\sqrt{2},
\sqrt{3})[[X]]=\mathbb{Q}+M$, where $M=X\mathbb{Q}(\sqrt{2},\sqrt{3})[[X]]$ and $V=\mathbb{Q}(\sqrt{2}, \sqrt{3})+M$. Clearly $M$ is a strongly stable ideal of $R$. Let $I$ be the ideal of $R$ given by $I=X(\mathbb{Q}+\sqrt{2}\mathbb{Q}+\sqrt{3}\mathbb{Q}+M)$. Then $I\subsetneq M$, but $I^{2}=M^{2}$. Thus $M$ is not a big ideal.
\end{example}

\begin{example} Let $k$ be a field, $X$ an indeterminate over $k$ and $R=k[[X^{2}, X^{3}]]$. Then $R$ is a one-dimensional local Noetherian divisorial domain (\cite[Theorem A]{Ba2} or \cite[Theorem 2.1]{Mi}) which is strongly stable, and so a big ideal domain. Indeed, let $M=(X^{2}, X^{3})R$ be the maximal ideal of $R$ and $I$ a nonzero ideal of $R$. If $II^{-1}=R$, then $I$ is a prncipal ideal of $R$ and so we are done. Assume that $II^{-1}\subseteq M$. Then $k[[X]]=M^{-1}\subseteq (II^{-1})^{-1}=(I_{v}:I_{v})=(I:I)\subseteq R'=k[[X]]$. Thus $(I:I)=k[[X]]$ and so $I$ is a principal ideal of $k[[X]]$. Thus $I$ is a strongly stable ideal of $R$. Therefore $R$ is a strongly stable domain and so a big ideal domain by Proposition~\ref{stable}.
\end{example}
\begin{example} A one-dimensional local Noetherian divisorial domain need not be a big ideal domain. Indeed,  let $k$ be a field, $X$ an indeterminate over $k$ and $R=k[[X^{3}, X^{4}]]$. Then $R$ is a one-dimensional local Noetherian divisorial domain (\cite[Theorem A]{Ba2} or \cite[Theorem 2.1]{Mi})  which is not a big ideal domain. Indeed, let $M=(X^{3}, X^{4})$ be the maximal ideal of $R$. Clearly $T=M^{-1}=k[[X^{3}, X^{4}, X^{5}]]$. Let $N=(X^{3}, X^{4}, X^{5})T=X^{3}k[[X]]$ be the maximal ideal of $T$ and $I=(X^{3}, X^{4})T$. Clearly $I\subsetneqq N$ and $I^{2}=N^{2}=X^{6}k[[X]]$. Thus $N$ is not a big ideal of $T$ and so $T$ is not a big ideal domain. By Remark~\ref{remark 1}, $R$ is not a big ideal domain since $(R:T)=M$. 
\end{example}
\begin{example} The following example illustrates Theorem~\ref{Pull.1}\1. It shows that  if $D$ and $T$ are big ideal domains, then $R$ is not necessarily a big ideal domain even in the case where $qf(D)=K$.\\  
Let $\mathbb{Q}$ be the field of rational numbers, $X$ an indeterminate over $\mathbb{Q}$, $T=\mathbb{Q}(\sqrt{2})+X\mathbb{Q}(\sqrt{2},
\sqrt{3})[[X]]=\mathbb{Q}(\sqrt{2})+M$ and $R=\mathbb{Z}[\sqrt{2}]+M$. Since $T$ is a $PVD$ (with associated valuation overring $V=\mathbb{Q}(\sqrt{2},
\sqrt{3})[[X]]=\mathbb{Q}(\sqrt{2}, \sqrt{3})+M$), $T$ is a big ideal domain by Corollary~\ref{Pull.2}. Clearly
$D=\mathbb{Z}[\sqrt{2}]$ is a big ideal domain and $qf(D)=K=T/M$. However $R$ is not a big ideal domain. Indeed, consider the ideals $I=X(\mathbb{Q}+\sqrt{2}\mathbb{Q}+\sqrt{3}\mathbb{Q}+M)\subsetneq M$. But $I^{2}=M^{2}$.
\end{example}

\begin{example} The following example illustrates Theorem~\ref{TRE.1}.\\
 Let $A=\mathbb{Z}$, $E=\mathbb{Z}_{2}=\mathbb{Z}/2\mathbb{Z}$, $I=2\mathbb{Z}$. Then, $I$ is a big ideal of $A$ but $I\ltimes E$ is not a big ideal.\\
\end{example}
\begin{example} The following example illustrates Theorem~\ref{TRE.4}.
Let $k$ be a field, $X$ and $Y$ indeterminates over $k$, $V=k(X)[[Y]]=k(X)+M$, $A=k+M$, $E=V/M$, $F=A/M$ and $R=A\ltimes E$. Clearly $M\ltimes F\subsetneqq M\ltimes E$ and $(M\ltimes F)^{2}=M^{2}\ltimes 0=(M\ltimes E)^{2}$. Thus $M\ltimes E$ is not a big ideal. However $M\ltimes E$ is an upper big ideal of $R$ as it is a mxaimal ideal.
\end{example}
\begin{example}\label{Example 7} A non-zero nilpotent ideal $I$  of a commutative ring $R$ is not a big ideal of $R$. However, it may be an upper big ideal. Indeed, Let $A=k$ be a field, $E$ a one-dimensional vector space over $k$ and $R:=k\ltimes E$. Since $dim_{k}E=1$, $E$ is a simple $A$-module, and clearly the only nonzero ideal of $R$ is $0\ltimes E$. As a maximal ideal of $R$, $0\ltimes E$ is an upper big ideal of $R$ which is not a big ideal (since it is a non-zero nilpotent ideal).
\end{example}
\section{\bf Declarations}

\begin{itemize}

\item{\bf Funding}  This research is supported by KFUPM under DSR Research Grant: SB191014

\item{\bf Conflict of interest/Competing interests}: The author states that there is no conflict of interest.

\item{\bf Availability of data and material}: Data sharing not applicable to this article as no datasets were generated or analysed during the current study

\end{itemize}
\noindent{\bf Acknowledgment}. The author would like to express his sincere thanks to Professor Muhammad Zafrullah for correcting the proofs of Theorems 2.7 and 2.11 in the earlier version.


\begin{thebibliography}{99}

\bibitem{AnKw} D. D. Anderson and D. J. Kwak, The $u,u^{-1}$ lemma revisited, Comm. Algebra \textbf{24} (7) (1996),  2447--2454.

\bibitem{AW} D. D. Anderson and M. Winders, Idealization of a module, J. Commut. Algebra 1 (1) (2009), 3--56.

\bibitem{ABDFK} D. F. Anderson, A. Bouvier, D. E. Dobbs, M. Fontana, and S. Kabbaj, On Jaffard domains, Expo. Math. \textbf{6} (1988) 145--175.



\bibitem{AD}    D. F. Anderson and D. E. Dobbs, Pairs of rings with the same prime ideals, Canad. J. Math. \textbf{32} (1980) 362--384.


\bibitem{Bar1} V. Barucci, On a class of Mori domains, Comm. Algebra \textbf{11} (17) (1983), 1989--2001.

\bibitem{Bar2} V. Barucci, Strongly divisorial ideals and complete integral closure of an integral domain, J. Algebra \textbf{99} (1986), 132–-142.

\bibitem{BarHou} V. Barucci and E. Houston, On the Prime Spectrum of a Mori Domain, Comm. Algebra \textbf{11} (24) (1996), 3599--3622.

\bibitem{BG}    E. Bastida and R. Gilmer, Overrings and divisorial ideals of rings of the form D+M, Michigan Math. J. \textbf{20} (1992) 79--95.



\bibitem{Ba2}    S. Bazzoni, Divisorial domains, Forum Math. \textbf{12} (2000) 397--419.



\bibitem{BR}   J.W. Brewer and E.A. Rutter, $D+M$ constructions with general overrings, Michigan Math. J. \textbf{23} (1976) 33--42.

\bibitem{CMZ} D. Costa,  J. L. Mott and M. Zafrullah,  The construction $D+XDs[X]$,  J. Algebra \textbf{53} (2) (1978), 423--439.












\bibitem{F}     M. Fontana, Topologically defined classes of commutative rings, Ann. Mat. Pura Appl. \textbf{123} (1980) 331--355.

\bibitem{FG}    M. Fontana and S. Gabelli, On the class group and the local class group of a pullback, J. Algebra  181 (3)  (1996) 803--835.





\bibitem{GH}    S. Gabelli and E. Houston, Coherentlike conditions in pullbacks, Michigan Math. J., \textbf{44} (1997), 99--123.


\bibitem{Gl} S. Glaz, \textit{Commutative coherent rings}, Springer-Verlag, Lecture Notes in Mathematics (1989), 13-71.


\bibitem{H1}    J. Hays, Reductions of ideals in commutative rings, Trans. Amer. Math. Soc. \textbf{177} (1973) 51--63.

\bibitem{H2}    J. Hays, Reductions of ideals in Pr\"ufer domains, Proc. Amer. Math. Soc. \textbf{52} (1975) 81--84.

\bibitem{HH}   J. Hedstrom and E. Houston, Pseudo-valuation domains, Pacific J. Math. \textbf{75} (1978) 137--147.


\bibitem{HLS}   W. Heinzer, D. Lantz and K. Shah, The Ratliff-Rush Ideals in a Noetherian ring, Comm. in Algebra \textbf{20} (1992), 591--622.














\bibitem{Kab}   S. Kabbaj, On the dimension theory of polynomial rings over pullbacks. In ``Multiplicative ideal theory in commutative algebra," pp. 263--277, Springer,  2006.

\bibitem{KLM}   S. Kabbaj, T. Lucas, and  A. Mimouni, Trace properties and pullbacks, Comm. Algebra \textbf{31} (3) (2003) 1085--1111.

\bibitem{KMa}   S. Kabbaj and  N. Mahdou, Trivial extensions defined by coherent-like conditions, Comm. Algebra \textbf{32} (10) (2004) 3937--3953.









\bibitem{Mi} A. Mimouni,  Note on the divisoriality of domains of the form $k[[X^p, X^q]], k[X^p, X^q], k[[X^p, X^q, X^r]]$ and $k[X^p, X^q, X^r]$, Turkish J. Math. 40 (1) (2016), 38--42.


\par\bibitem{Na}    M. Nagata, Local rings, Interscience Tracts in Pure and Applied Mathematics, No. 13, Interscience Publishers, New York-London, 1962.

\bibitem{NR}    D. G. Northcott and D. Rees, Reductions of ideals in local rings, Proc. Cambridge Philos. Soc. \textbf{50} (1954) 145--158.




\bibitem{RR}    L. J. Ratliff, Jr and D. E. Rush, Two notes on reductions of ideals, Indiana Univ. Math. J. \textbf{27} (1978) 929--934.



\bibitem{Sharma} P. K. Sharma, Ideal containment vs. Powers, arXiv: 1903.11035v1 [math.AC] 26 Mar 2019.



\bibitem{Zaf1} M. Zafrullah, The $D+XD_{S}[X]$ construction from $GCD$-domains, J. Pure Appl. Algebra \textbf{50} (1) (1988), 93--107.

\bibitem{Zaf2} M. Zafrullah,  Facets on rings between $D[X]$ and $K[X]$, Commutative ring theory and applications (Fez, 2001), 445–460, Lecture Notes in Pure and Appl. Math., 231, Dekker, New York, 2003.

\bibitem{Zaf3} M. Zafrullah, Various facets of rings between $D[X]$  and $K[X]$,  Comm. Algebra \textbf{31} (5) (2003), 2497--2540. 
\end{thebibliography}
\end{document}